\documentclass[12pt,a4paper,oneside,reqno,notitlepage]{amsart}

\usepackage{amsfonts}
\usepackage{amscd,amsmath}
\usepackage{euscript}

\usepackage{amssymb,latexsym}
\usepackage{graphics}

\usepackage[arrow,matrix,curve]{xy}\SilentMatrices
\def\xyma{\xymatrix@M.7em}

\numberwithin{equation}{section}

\usepackage[T2A]{fontenc}

\newtheorem{prop}{Proposition}[section]
\newtheorem{theorem}{Theorem}[section]
\newtheorem{lemma}{Lemma}[section]
\newtheorem{remark}{Remark}[section]

\def\bee{\begin{equation}}
\def\ee{\end{equation}}

\def\f2{\mathbb{F}_2}

\def\bee{\begin{equation}}
\def\ee{\end{equation}}

\begin{document}
\title{Narain Gupta's three normal subgroup problem and group homology}
\author{Roman Mikhailov and Inder Bir S. Passi}

\maketitle

\centerline{\it Dedicated to the memory of Chander Kanta Gupta and Narain Gupta}

\begin{abstract}
This paper is about application of various homological methods  to classical problems in the theory of group rings. It is shown  that the third homology of  groups plays a key role in Narain Gupta's three normal subgroup problem. For a free group $F$ and its normal subgroups $R,\,S,\,T,$ and the corresponding ideals in the integral group ring $\mathbb Z[F]$, ${\bf r}=(R-1)\mathbb Z[F],\ {\bf s}=(S-1)\mathbb Z[F],\ {\bf t}=(T-1)\mathbb Z[F],$  a complete description of the normal  subgroup $F\cap (1+{\bf rst})$ is given, provided $R\subseteq T$ and the third and the fourth homology groups of $R/R\cap S$ are torsion groups.
\end{abstract}

\section{Introduction}

It is well-known that the second (co)homology of  groups plays an important role in the theory of groups; in particular, in the theory of central extensions.
The third cohomology of a group classifies  $k$-invariants for crossed modules or homotopy 2-types (\cite{Holt:1979}, \cite{Huebschmann:1980}, \cite{Ratcliffe:1980}). However, it is not easy to find an explicit application of the third (co)homology in group-theoretical questions which are formulated without the language of homological algebra. In this paper, we show that the third homology of groups plays a key role in the solution of a problem in free group rings.
\par\vspace{.5cm}
Let $F$ be a free group and  $\mathbb Z[F]$  its integral group ring. For every two-sided ideal $\mathfrak a$ in $\mathbb Z[F]$, we have a normal subgroup $D(F,\,\mathfrak a):=F\cap (1+ \mathfrak a)$ of $F$. The identification  of such normal subgroups in free groups is a recurring problem in the theory of group rings (see \cite{Gupta}, \cite{MP:2009}, \cite{Passi:1979}). As demonstrated in our works (\cite{MP:2015a}, \cite{MP:2016}, \cite{MP:2016a}), homology of groups and derived functors of non-additive functors can provide a useful tool for investigating such subgroups. In the present article we use this homological approach to address Narain Gupta's problem (\cite{Gupta}, Problem 6.3, p.\,119) in free group rings  which, in general,  has been rather intractable so far.

\par\vspace{.5cm}
 Given  a normal subgroup $R$ of $F$, let $\mathbf r$ denote the two-sided ideal of $\mathbb Z[F]$ generated by the augmentation ideal $\Delta(R)$ of the integral group ring $\mathbb Z[R],$ i.e., $\mathfrak r=\Delta(R)\mathbb Z[F]$.  Clearly $D(F,\,{\bf r})=R$. For two normal subgroups $R,\,S$ of $F$, it is known that $D(F,\,{\bf rs})=[R\cap S,\,R\cap S]$, the derived subgroup of $R\cap S$ (\cite{Enright:1968}; \cite{Gupta}, Theorem 1.6, p.\,3).  If $R,\,S,\,T $ are three normal subgroups of $F$, a  currently open problem formulated by Narain Gupta (loc. cit.) asks for the identification of the normal subgroup $D(F,\,\mathbf r\mathbf s\mathbf t)$.
The answer to this general problem is known for the following special cases:\pagebreak \par\vspace{.25cm}\noindent
\begin{quote}
\begin{equation}\label{c1} D(F,\, {\bf rfr}) = \gamma_3(R) \  \text{(\cite{Kanta:1978};  \cite{Gupta}, p.\,116),}\end{equation}
\par\vspace{.5cm}\noindent
\begin{equation}\label{c2}
 D(F,\,{\bf rrs}) =[R'\cap S,\,R'\cap S]\gamma_3(R\cap S) \ \text{\cite{KKV}},\end{equation}\par\vspace{.5cm}\noindent
\begin{equation}\label{c3} D(F,\,{\bf rfs}) =\sqrt{[R'\cap S',\,R\cap S][R\cap S',\,R\cap S'][R'\cap S,\, R'\cap S]}\ \ \text{(\cite{Kanta:1978},\ \cite{MP:2016});}\end{equation}\par\vspace{.5cm}\noindent
\begin{equation}\label{c4} D(F,\,{\bf frf})=\sqrt{[R',\,F]}\  \text{\cite{Stohr:1984},}\end{equation}\par\vspace{.25cm}\noindent

\end{quote}  where, for groups $H\subseteq G$, $H'$ and $\sqrt{H} $ denote respectively the derived subgroup of $H$ and the isolator in $G$ of the subgroup $H$, and $\{\gamma_i(G)\}_{i\geq 1}$ is the lower central series of $G$.
Given  a triple of subgroups $R,\,S,\,T$ of a free group $F$, set
$$
I(R,\,S,\,T):=\sqrt{[(R\cap S)'\cap (S\cap T)',\, R\cap T](R\cap (S\cap
T)')'((R\cap S)'\cap T)' }.
$$
Observe that, for all the above-mentioned known identifications of $D(F,\, {\bf rst})$, we have
\begin{equation}\label{conj}
D(F,\,{\bf rst})=I(R,\,S,\,T).
\end{equation} \par\vspace{.5cm}
The object of the present work  is to investigate  the case when $R\subseteq T$ (or equivalently, in view of the canonical anti-automorphism  of $\mathbb Z[F]$,  when $T\subseteq R$). It is easy to see that a complete answer for this case, together with the  known results (\ref{c2},\,\ref{c3},\,\ref{c4}), will provide identification of $D(F,\,{\bf rst})$ whenever one of the three normal subgroups $R,\,S,\,T$ is contained in either of the other two. Our main result is the following
\par\vspace{.25cm}\noindent
\begin{theorem}\label{rsf}
If $R,\ S,\ T$ are normal  subgroups of a free group $F$, such that $R\subseteq T$, and the integral homology groups $H_3(R/R\cap S)$,  $H_4(R/R\cap S)$ are torsion groups, then
$$
D(F, \, {\bf rst})=I(R,\,S,\,T)=\sqrt{(R\cap (S\cap T)')'[(R\cap S)',\, R]}.
$$
\end{theorem}\par\vspace{.25cm}
Our proof of the above theorem involves a mix of  homological and combinatorial \linebreak arguments, which we develop in Section 2, and it is completed in Section 3. A  striking feature to note here is the role played by the third homology in the identification of normal subgroups determined by ideals in free group rings.   In Section 4 we bring out further the role of integral homology and give an example with $D(F, \,{\bf rsf})/I(R, \,S,\,F)$ non-zero, thus showing that (\ref{conj}) does not hold in general.
\par\vspace{.5cm}
In Section 5, we prove (Theorem \ref{Stohr1}), using combinatoral arguments,  that if the normal subgroup $R$ is contained in both $S$ and $T$, and
$$
a\in D(F,\,{\bf srt+trs}),
$$
then
$$
a^2\in D(F,\,{\bf rrs+srr+trr+rrt})
$$
and therefore, by \cite{Stohr:1984}, Theorem 4, $a^4\in [R,\,R,\,ST].$ Thus, in particular, we have a combinatorial proof  of one of Ralph St\"{o}hr's  results, which is implicit in his  homological approach  \cite{Stohr:1984} to  Gupta's problem, namely that if $a\in D(F,\,{\bf frf})$, then $a^2\in D({\bf rrf+frr})$.
\par\vspace{.5cm}
 We conclude with a few observations on the corresponding four normal subgroup\linebreak  problem including the identification
$$
D(F,\, {\bf rsfr})=D(F, \,{\bf rfsr})=[R\cap S',\, R\cap S',\, R]\gamma_4(R),
$$provided $R\subseteq S$, which is a generalization of a result of Chander Kanta Gupta \cite{Kanta:1983}.

.
\par\vspace{.5cm}
\section{Homological and Combinatorial Preliminaries}\par\vspace{.5cm}
\begin{theorem}\label{relative}
If  $F$ is a free group, and $R,\,S$ its normal subgroups with $S\subseteq R$, then there is a natural isomorphism
$$
H_2(F/S,\, {\bf f}/{\bf r})\cong \frac{{\bf fs}\cap {\bf rf}}{{\bf fsf}+{\bf rs}}.
$$
\end{theorem}
\begin{proof}
Consider the Gruenberg free resolution (\cite{Gruenberg:1970}, p.\,34)$$
\cdots\to \mathbf s\mathbf f/\mathbf s^2\mathbf f\to \mathbf s/\mathbf s^2\to \mathbf f/\mathbf s\mathbf f\to \mathbb Z[F/S]\to \mathbb Z\to 0$$of $\mathbb Z$ viewed as a trivial left $F/S$-module. On tensoring this resolution with the right $F/S$-module $\mathbf  f/\mathbf r$, we have the complex $$
\cdots\to \mathbf f/\mathbf r\otimes_{F/S}\mathbf s\mathbf f/\mathbf s^2\mathbf f\to
\mathbf f/\mathbf r\otimes_{F/S}\mathbf s/\mathbf s^2\to\mathbf f/\mathbf r\otimes_{F/S}\mathbf f/\mathbf s\mathbf f\to \mathbf f/\mathbf r.$$
For a free group $F$, and ideals $\mathfrak b\subset \mathfrak a, \ \mathfrak d\subset \mathfrak c$, we have (\cite{IM}, Lemma 4.9) $$(\mathfrak a/\mathfrak b)\otimes _F(\mathfrak c/\mathfrak d)\cong \frac{\mathfrak a\mathfrak c}
{\mathfrak b\mathfrak c+\mathfrak a\mathfrak d}.$$
Thus the above complex reduces to the following complex
$$
\cdots\to \frac{ \mathbf f\mathbf s\mathbf f}{\mathbf f\mathbf s^2\mathbf f+\mathbf r\mathbf s\mathbf f}\to \frac{
\mathbf f\mathbf s}{\mathbf f\mathbf s^2+\mathbf r\mathbf s}\to\frac{\mathbf f^2}{\mathbf f\mathbf s\mathbf f+\mathbf r\mathbf f}\to \mathbf f/\mathbf r.$$
Hence $$H_2(F/S,\, {\bf f}/{\bf r})\cong \frac{{\bf fs}\cap {\bf rf}}{{\bf fsf}+{\bf rs}}.
$$
\end{proof}
\begin{remark}\label{rem1} Note that the  natural composition
$$
\xyma{
H_2(R/S)\otimes {\bf f}/{\bf r}\ar@{>->}[r] \ar@{=}[d] & H_2(R/S,\, {\bf f}/{\bf r}) \ar@{->}[r] & H_2(F/S,\, {\bf f}/{\bf r})\ar@{>->}[d]\\ \frac{\Delta^2(R)\cap \Delta(S)\mathbb Z[R]}{\Delta(R)\Delta(S)+\Delta(S)\Delta(R)}\otimes{\bf f}/{\bf r} \ar@{->}[rr]
 & & \frac{\bf fs}{\bf fsf+rs}}
$$
is induced by the map
$$
(\alpha,\, \beta)\mapsto \beta\alpha,\, \alpha\in \Delta^2(R)\cap \Delta(S)\mathbb Z[R],\ \beta\in {\bf f}.
$$
\end{remark}
\par\vspace{.25cm}

 \begin{lemma}\label{lemma1.1}  $($see \cite{VR}, Theorem 1.1$)$ If $R,\,S,\,T $ are normal subgroups of a free group $F$ and $R\subseteq T$, then
\begin{equation}\label{lemma}\Delta(R)\Delta(S)\Delta(T)\cap \Delta(R)\Delta(R\cap S) =
\Delta(R)\Delta(R\cap S)\Delta(R)+\Delta(R)\Delta(R\cap (S\cap T)'). \end{equation}\end{lemma}
\par\vspace{.25cm}
\begin{proof} We can assume that $F=ST=TS$.
Let $$u\in \Delta(R)\Delta (S)\Delta(T)\cap \Delta(R)\Delta(R\cap S).$$
Since the ideal $\Delta(R)\mathbb Z[F]$ is a free ringht $\mathbb Z[F]$-module,
\begin{equation}\label{e1}
u=\sum _{y\in \mathcal Y}(y-1)u_y,\end{equation}
where $\mathcal Y$ is a free basis of $R$
and
\begin{equation}\label{e2}
u_y\in \Delta(S)\Delta(T)\mathbb Z[F]\cap \Delta(R\cap S)\mathbb Z[F]
=\Delta(S)\Delta(T)\cap \mathbb Z[F]\Delta(R\cap S).
\end{equation}
Let $\mathcal Z$ be a transversal for $T$ in $F=ST$.
Since $\Delta(R)\Delta(R\cap S)\subseteq \mathbb Z[T]$, projecting the above equation (\ref{e1}) under the map $\theta:\mathbb Z[F]\to \mathbb Z[T]$ induced by $$f=ts\mapsto t$$for $f\in F.\ t\in T,\ s$ in a  right transversal for $S\cap T$ in $S$, it follows that \begin{equation}\label{e3}u_y\in \mathbb Z[T].\end{equation}

Similarly, using the projection $\mathbb Z[F]\to \mathbb Z[T]$ induced on using left transversal for $S\cap T$ in $S$, the inclusion (\ref{e2}) shows further that
\begin{equation}\label{e4}
u_y\in \Delta(S\cap T)\Delta(T)\cap \mathbb Z[T]\Delta(R\cap S).\end{equation}
Therefore, $u_y\in \Delta(R\cap S)\Delta(T)+\Delta(R\cap(S\cap T)')$ [Lemma 2.1 \cite{VR}, by setting $K=S\cap T$, $H=R$]. Consequently
\begin{equation}\label{e5}
u\in \Delta(R)\Delta(R\cap S)\Delta(T)+\Delta(R)\Delta(R\cap(S\cap T)').\end{equation} Because $u\in \Delta(R)$, projecting (\ref{e5}) under $\mathbb Z[T]\to \mathbb Z[R]$, it follows that
\begin{equation}\label{e6}
   u\in \Delta(R)\Delta(R\cap S)\Delta(R)+\Delta(R)\Delta(R\cap(S\cap T)').                                                                                                                                                                                                                                                                                                                                                                                                                                                                                                                                                                             \end{equation}
We have thus proved that the left hand side  of (\ref{lemma}) is contained in the right hand side; the reverse inclusion being obvious, the proof of the Lemma is complete. \end{proof}
\par\vspace{.5cm}
A similar analysis, as above, yields the following intersection lemma.
\par\vspace{.5cm}
\begin{lemma}\label{onemorelemma}
Let $R\subseteq S$. Then
$$
\Delta(R)\Delta(F)\Delta(S)\Delta(R)\cap \Delta^3(R)=\Delta^4(R)+\Delta(R)\Delta(R\cap S')\Delta(R). \quad \Box
$$
\end{lemma}

\par\vspace{.5cm}
For a free group $F$ and its normal subgroup $R$, and any  left $\mathbb Z[F/R]$-module $M$, there are isomorphisms
\begin{equation}\label{dimshift}
H_i(F/R, \,R/R'\otimes M)\cong H_{i+2}(F/R,\,M),\ i\geq 1.
\end{equation}
This is a well-known fact which  follows easily from the Magnus embedding
$R/R'\hookrightarrow { \bf f/\bf f\bf r}$  of the relation module $R/R'$.
\par\vspace{.5cm}
Let $R,\,S$ be normal subgroups of $F$. For the group $G:=R/R\cap S\cong RS/S$, one can consider two different relation modules and a natural map between them:
$$
\frac{R\cap S}{(R\cap S)'}\to S/S'.
$$
This map can be naturally extended to a map between the corresponding relation \linebreak sequences (\cite{Gruenberg:1976}, p.\,7):
$$\xyma{\frac{R\cap S}{(R\cap S)'}
\ar@{>->}[r]\ar@{->}[d] &\frac{\Delta(R)}{\Delta(R\cap
S)\Delta(S)}\ar@{->>}[r]\ar@{->}[d] & \frac{\Delta(R)}{\Delta(R\cap S)\mathbb Z[R]}\ar@{->}[d]^\simeq\\
S/S' \ar@{>->}[r] &
\frac{\Delta(RS)}{\Delta(S)\Delta(RS)}\ar@{->>}[r] & \frac{\Delta(RS)}{\Delta(S)\mathbb Z[RS]}}$$
The $\mathbb Z[G]$-modules $\frac{\Delta(R)}{\Delta(R\cap
S)\Delta(S)}$ and $\frac{\Delta(RS)}{\Delta(S)\Delta(RS)}$ are free, hence for any $\mathbb Z[G]$-module $G$, there are natural isomorphisms
\begin{equation}\label{shift}
H_i\left(R/R\cap S,\, \frac{R\cap S}{(R\cap S)'}\otimes M\right)\cong H_i(R/R\cap S, \,S/S'\otimes M)\cong H_{i+2}(R/R\cap S,\,M),\ i\geq 1.
\end{equation}
\par\vspace{.5cm}
We need  some well-known facts about certain quadratic endofunctors on the category of abelian groups, namely:
\begin{align*}
& \otimes^2\ {\text tensor\ square},\\
& {\sf SP^2}\ {\text symmetric\ square},\\
& \Lambda^2\ {\text exterior\ square},\\
& \tilde\otimes^2\ {\text antisymmetric\ square},\\
& \Gamma^2\ {\text divided\ square}.
\end{align*}
For a survey of the properties of these functors and their derived functors, see (\cite{BP}, \cite{BM}).  Recall that, for an abelian group $A$, by definition,
\begin{align*}
& {\sf SP^2}(A)=\otimes^2(A)/\langle a\otimes b-b\otimes a,\ a,b\in A\rangle,\\
& \Lambda^2(A)=\otimes^2(A)/\langle a\otimes a,\ a\in A\rangle,\\
& \tilde\otimes^2(A)=\otimes^2(A)/\langle a\otimes b+b\otimes a,\ a,b\in A\rangle.
\end{align*}
The divided square functor $\Gamma^2$  is also known as the Whitehead quadratic functor. Given an abelian group $A$, the abelian group $\Gamma^2(A)$ is generated by symbols $\gamma(x),$ $x\in A$, satisfying the following relations for all $x,\, y,\,z\in  A$:
\begin{align*}
& \gamma(0)=0;\\
& \gamma(x) = \gamma(-x);\\
& \gamma(x+y+z)-\gamma(x+y)-\gamma(x+z)-\gamma(y+z)+\gamma(x)+\gamma(y)+\gamma(z)=0.
\end{align*}
\par\vspace{.5cm}\noindent
The exterior and the antisymmetric squares are connected as follows. For every abelian group $A$, we have  a short exact sequence
\begin{equation}
0\to A\otimes \mathbb Z/2\to \tilde\otimes^2(A)\to \Lambda^2(A)\to 0.
\end{equation}
Similarly, connecting the  symmetric and divided square functor, we have the following short exact sequence:
\begin{equation}
0\to {\sf SP}^2(A)\to \Gamma^2(A)\to A\otimes \mathbb Z/2\to 0.
\end{equation}

Let $E$ be a free abelian group, $I$ its subgroup and $E/I=A$. Then there is a natural exact sequence
\begin{equation}\label{derived1}
0\to L_1{\sf SP^2}(A)\to \Lambda^2(E)/\Lambda^2(I)\to E\otimes E/I\to {\sf SP}^2(A)\to 0
\end{equation}
where $L_1{\sf SP}^2$ is the first derived functor of ${\sf SP}^2$ in the sense of Dold-Puppe, and $L_1{\sf SP}^2(A) $ is  equal to the quotient of ${\sf Tor}(A,\,A)$ by the subgroup generated by the diagonal elements. We refer the reader to  \cite{MP:2016} and \cite{MP:2016a} for the proof and applications of above  kind of sequences in the theory of groups rings.
\par\vspace{.5cm}Another ingredient that we need is the following analog of the results from \cite{Kock} on Koszul sequences.
\par\vspace{.25cm}\noindent
\begin{lemma}\label{koszultype}
For a free abelian group $E$, $I\subseteq E,\ E/I=A$, the homology of the naturally defined complex
$$
{\sf SP}^2(I) \rightarrowtail I\otimes E\to \tilde\otimes^2(E)
$$
satisfies the following:
$$
H_0=\tilde\otimes^2(A),
$$
\begin{equation}\label{l1to}
0\to {\sf Tor}(A,\,\mathbb Z/2)\to H_1\to L_1\Lambda^2(A)\to 0.
\end{equation}
\end{lemma}
\begin{proof}
The description of $H_0$ follows from the natural commutative diagram
$$
\xyma{& {\sf SP^2}(E)\ar@{>->}[d] \ar@{=}[r] & {\sf SP^2(E)}\ar@{->}[d]\\
I\otimes E \ar@{=}[d] \ar@{>->}[r] & \otimes^2(E)\ar@{->>}[r]\ar@{->>}[d] & \otimes^2(A)\ar@{->>}[d]\\ I\otimes E\ar@{->}[r] & \tilde\otimes^2(E)\ar@{->>}[r] & \tilde\otimes^2(A)}
$$
since the image of ${\sf SP}^2(E)$ in $\otimes^2(A)$ is the same as the image of ${\sf SP}^2(A)$ in $\otimes^2(A)$.
\par\vspace{.5cm}
Recall that the first homology of the Koszul sequence
$$
\Gamma^2(I)\to I\otimes E\to \Lambda^2(E)
$$
is naturally isomorphic to the derived functor $L_1\Lambda^2(A)$ (see \cite{Kock}). Observe that
the kernel of the map $I\otimes \mathbb Z/2\to E\otimes \mathbb Z/2$ is naturally isomorphic to ${\sf Tor}(A,\,\mathbb Z/2)$.
The short exact sequence (\ref{l1to}) follows from the following commutative diagram:
$$
\xyma{& & E\otimes \mathbb Z/2\ar@{->>}[r] \ar@{>->}[d] & A\otimes \mathbb Z/2\ar@{>->}[d]\\
{\sf SP}^2(I)\ar@{>->}[r] \ar@{>->}[d] & I\otimes E\ar@{->}[r]\ar@{=}[d] & \tilde\otimes^2(E)\ar@{->>}[r] \ar@{->>}[d] & \tilde\otimes^2(A)\ar@{->>}[d]\\
\Gamma^2(I)\ar@{->>}[d] \ar@{>->}[r] & I \otimes E\ar@{->}[r] & \Lambda^2(E)\ar@{->>}[r] & \Lambda^2(A)\\ I\otimes \mathbb Z/2.}
$$
\end{proof}\par\vspace{.5cm}
\section{Proof of Theorem \ref{rsf}}
\par\vspace{.5cm}

Since $$F\cap (1+\Delta(R)\Delta(S))=[R\cap
S,R\cap S]\subset (1+\Delta(R\cap S)\Delta(R)),$$ Lemma \ref{lemma1.1} implies that
$$
D:=D(F,\,\Delta(R)\Delta(S)\Delta(F))=F\cap (1+\Delta(R)\Delta(R\cap S)\Delta(R)+\Delta(R)\Delta(R\cap (S\cap T)'))
$$
Observe that, for $w\in D$, $w-1\in \Delta(R\cap S)\Delta(R),$ and  $$\Delta(R)\Delta(R\cap S)\Delta(R)\subset \Delta(R\cap S)\Delta(R).$$ Therefore
\begin{equation}\label{d}
D=F\cap (1+\Delta(R)\Delta(R\cap S)\Delta(R)+\Delta(R)\Delta(R\cap (S\cap T)')\cap \Delta(R\cap S)\Delta(R))
\end{equation}
We have to prove that the quotient
\begin{equation}\label{firstquotient}
\frac{F\cap (1+\Delta(R)\Delta(R\cap S)\Delta(R)+\Delta(R)\Delta(R\cap (S\cap T)')\cap \Delta(R\cap S)\Delta(R))}{(R\cap (S\cap T)')'[(R\cap S)',\,R]}
\end{equation}
is a torsion group.
By Theorem \ref{relative},
\begin{equation}\label{h2rs}
\frac{\Delta(R)\Delta(R\cap (S\cap T)')\cap \Delta(R\cap S)\Delta(R)}{\Delta(R)\Delta(R\cap (S\cap T)')\Delta(R)+\Delta(R\cap S)\Delta(R\cap (S\cap T)')}=H_2\left(\frac{R}{R\cap (S\cap T)'}, \,\frac{\Delta(R)}{\Delta(R\cap S)}\right).
\end{equation}
Consider the homological Hochschild-Serre spectral sequence for the group extension
$$
1\to \frac{R\cap S}{R\cap (S\cap T)'}\to \frac{R}{R\cap (S\cap T)'}\to \frac{R}{R\cap S}\to 1
$$
with coefficients in $\frac{\Delta(R)}{\Delta(R\cap S)}$. Three terms of the $E^2$-page of the spectral sequence, which give a contribution to the homology group (\ref{h2rs}), are the following:\\
\begin{align*}
& E_{0,\,2}^2=H_0\left(\frac{R}{R\cap S},\, H_2\left(\frac{R\cap S}{R\cap (S\cap T)'}, \frac{\Delta(R)}{\Delta(R\cap S)}\right)\right);\\\\
& E_{1,\,1}^2=H_1\left(\frac{R}{R\cap S},\, H_1\left(\frac{R\cap S}{R\cap (S\cap T)'}, \,\frac{\Delta(R)}{\Delta(R\cap S)}\right)\right);\\\\
& E_{2,\,0}^2=H_2\left(\frac{R}{R\cap S},\, \frac{\Delta(R)}{\Delta(R\cap S)}\right)=H_3\left(\frac{R}{R\cap S}\right).
\end{align*}
By hypothesis, $E_{2,0}^2$ is a torsion group.  Consider the term $E_{1,1}^2$:\\
$$
E_{1,\,1}^2=H_1\left(\frac{R}{R\cap S}, \frac{R\cap S}{R\cap (S\cap T)'}\otimes \frac{\Delta(R)}{\Delta(R\cap S)}\right)=H_2\left(\frac{R}{R\cap S}, \,\frac{R\cap S}{R\cap (S\cap T)'}\right).
$$\\
The exact sequence
\begin{equation}\label{ses}
0\to \frac{R\cap S}{R\cap (S\cap T)'}\to \frac{S\cap T}{(S\cap T)'}\to \frac{S}{R\cap S}\to 0,
\end{equation}
yields an exact sequence\\
\begin{equation}\label{rleft}
H_3\left(\frac{R}{R\cap S}, \frac{S}{R\cap S}\right)\to E_{1,1}^2\to H_2\left(\frac{R}{R\cap S}, \frac{S\cap T}{(S\cap T)'}\right).
\end{equation}\\
The natural isomorphism\\
$$
\frac{R(S\cap T)}{S\cap T}\to \frac{R}{R\cap S}
$$
induces an isomorphism (see (\ref{shift}))
$$\\
H_4\left(\frac{R}{R\cap S}\right)\cong H_2\left(\frac{R}{R\cap S},\, \frac{S\cap T}{(S\cap T)'}\right).
$$
By hypothesis, all terms in (\ref{rleft}) are torsion. Consider the natural map
\begin{multline*}
f: H_2\left(\frac{R\cap S}{R\cap (S\cap T)'},\, \frac{\Delta(R)}{\Delta(R\cap S)}\right)=H_2\left(\frac{R\cap S}{R\cap (S\cap T)'}\right)\otimes \frac{\Delta(R)}{\Delta(R\cap S)}\to\\  H_2\left(\frac{R}{R\cap (S\cap T)'}, \frac{\Delta(R)}{\Delta(R\cap S)}\right).
\end{multline*}
The natural image of the map $f$  under the isomorphism (\ref{h2rs}) is generated by elements of the type $(f-1)([r_1,\,r_2]-1),$ with $ f\in R,\, r_1,\,r_2\in R\cap S$ (see Remark \ref{rem1}). It follows  that the quotient group
$$
\frac{\Delta(R)\Delta(R\cap (S\cap T)')\cap \Delta(R\cap S)\Delta(R)}{im(f)+\Delta(R)\Delta(R\cap (S\cap T)')\Delta(R)+\Delta(R\cap S)\Delta(R\cap (S\cap T)')}
$$
is a torsion group. Since $im(f)\subseteq \Delta(R)\Delta(R\cap S)\Delta(R)$, the quotient
$$
\frac{\Delta(R)\Delta(R\cap S)\Delta(R)+\Delta(R)\Delta(R\cap (S\cap T)')\cap \Delta(R\cap S)\Delta(R)}{\Delta(R)\Delta(R\cap S)\Delta(R)+\Delta(R\cap S)\Delta(R\cap (S\cap T)')}
$$
is a torsion group. Therefore, the quotient (\ref{firstquotient}) is torsion if and only if the quotient
\begin{equation}\label{secondquotient}
\frac{F\cap (1+\Delta(R)\Delta(R\cap S)\Delta(R)+\Delta(R\cap S)\Delta(R\cap (S\cap T)'))}{(R\cap (S\cap T)')'[(R\cap S)',R]}
\end{equation}
is torsion.
Since $F\cap (1+\Delta(R)\Delta(R\cap S)\Delta(R)+\Delta(R\cap S)\Delta(R\cap (S\cap T)'))\subseteq (R\cap S)',$ and $\Delta(R\cap S)\Delta(R\cap (S\cap T)')\subset \Delta^2(R\cap S),$
\begin{multline*}F\cap (1+\Delta(R)\Delta(R\cap S)\Delta(R)+\Delta(R\cap S)\Delta(R\cap (S\cap T)'))=\\
F\cap (1+\Delta(R)\Delta(R\cap S)\Delta(R)\cap \Delta^2(R\cap S)\mathbb Z[R]+\Delta(R\cap S)\Delta(R\cap (S\cap T)')).
\end{multline*}
Recall that
$$
H_4\left(\frac{R}{R\cap S}\right)=\frac{\Delta(R)\Delta(R\cap S)\Delta(R)\cap \Delta^2(R\cap S)\mathbb Z[R]}{\Delta^2(R\cap S)\Delta(R)+\Delta(R)\Delta^2(R\cap S)}
$$
is finite by hypothesis. Therefore, the quotient (\ref{secondquotient}) is torsion if and only if the quotient
$$
\frac{F\cap (1+\Delta^2(R\cap S)\Delta(R)+\Delta(R)\Delta^2(R\cap S)+\Delta(R\cap S)\Delta(R\cap (S\cap T)'))}{(R\cap (S\cap T)')'[(R\cap S)',R]}
$$
is torsion.
\par\vspace{.5cm} We next observe that the following commutative diagram with exact rows follows from (\ref{derived1}) (with $E=(R\cap S)/(R\cap S)'$, $I=(R\cap (S\cap T)')/(R\cap S)'\cap (S\cap T)'$, and the fact that the derived functor $L_1{\sf SP}^2$ in (\ref{derived1}) vanishes since the quotient $(R\cap S)/(R\cap (S\cap T)'\subseteq (S\cap T)/(S\cap T)'$ is torsion-free):
$$
\xyma{\frac{\Lambda^2\left(\frac{R\cap S}{(R\cap S)'}\right)}{\Lambda^2\left(\frac{R\cap (S\cap T)'}{(R\cap S)'\cap (S\cap T)'}\right)}\ar@{=}[d] \ar@{>->}[r] & \frac{R\cap S}{R\cap (S\cap T)'}\otimes \frac{R\cap S}{(R\cap S)'}\ar@{->>}[r] \ar@{=}[d] & {\sf SP}^2\left(\frac{R\cap S}{R\cap (S\cap T)'}\right)\ar@{=}[d]\\
\frac{(R\cap S)'}{(R\cap (S\cap T)')'[(R\cap S)',R\cap S]} \ar@{>->}[r] & \frac{\Delta^2(R\cap S)}{\Delta^3(R\cap S)+\Delta(R\cap S)\Delta(R\cap (S\cap T)')}\ar@{->>}[r] & {\sf SP}^2\left(\frac{R\cap S}{R\cap (S\cap T)'}\right).
}
$$
On passing to the homology $H_*(R/R\cap S, \,-)$, we obtain the exact sequence
\begin{multline}\label{3step}
H_1\left(R/R\cap S, \frac{R\cap S}{R\cap (S\cap T)'}\otimes \frac{R\cap S}{(R\cap S)'} \right)\to \\\\
H_1\left(R/R\cap S, {\sf SP}^2\left(\frac{R\cap S}{R\cap (S\cap T)'}\right)\right)\to \frac{(R\cap S)'}{(R\cap (S\cap T)')'[(R\cap S)',R]}\to\\\\ \frac{\Delta^2(R\cap S)}{\Delta^2(R\cap S)\Delta(R)+\Delta(R)\Delta^2(R\cap S)+\Delta(R\cap S)\Delta(R\cap (S\cap T)')}
\end{multline}
That is, we get an epimorphism
\begin{multline}\label{epi10}
H_1\left(R/R\cap S, {\sf SP}^2\left(\frac{R\cap S}{R\cap (S\cap T)'}\right)\right)\to\\\\ \frac{F\cap (1+\Delta^2(R\cap S)\Delta(R)+\Delta(R)\Delta^2(R\cap S)+\Delta(R\cap S)\Delta(R\cap (S\cap T)'))}{(R\cap (S\cap T)')'[(R\cap S)',R]}.
\end{multline}
The next commutative diagram follows from Lemma \ref{koszultype}:
$$
\xyma{& {\sf SP}^2\left(\frac{R\cap S}{R\cap (S\cap T)'}\right)\ar@{>->}[r] \ar@{>->}[d] & \frac{S\cap T}{(S\cap T)'}\otimes \frac{R\cap S}{R\cap (S\cap T)'}\ar@{>->}[d] \ar@{->}[r] & \tilde\otimes^2\left(\frac{S\cap T}{(S\cap T)'}\right)\ar@{=}[d]\\
& {\sf SP}^2\left(\frac{S\cap T}{(S\cap T)'}\right)\ar@{->>}[d] \ar@{>->}[r] & \frac{S\cap T}{(S\cap T)'}\otimes \frac{S\cap T}{(S\cap T)'}\ar@{->>}[r] \ar@{->>}[d] & \tilde \otimes^2\left(\frac{S\cap T}{(S\cap T)'}\right)\\
K\ar@{>->}[r] & \frac{{\sf SP}^2\left(\frac{S\cap T}{(S\cap T)'}\right)}{{\sf SP}^2(\frac{R\cap S}{R\cap (S\cap T)'})} \ar@{->}[r] & \frac{S\cap T}{(S\cap T)'}\otimes \frac{S\cap T}{(R\cap S)(S\cap T)'}\ar@{->>}[r] & \tilde\otimes^2\left(\frac{S\cap T}{(R\cap S)(S\cap T)'}\right), }
$$
where $K$ lives in the short exact sequence
$$
0\to {\sf Tor}\left(\frac{S\cap T}{(S\cap T)'(R\cap S)},\,\mathbb Z/2\right)\to K\to L_1\Lambda^2\left(\frac{S\cap T}{(S\cap T)'(R\cap S)}\right)\to 0;
$$
in particular, $K$ is a torsion group. Recall that the homology in dimension $\geq 1$ of a group with coefficients in the symmetric (or exterior, or antisymmetric) square of its relation module is a 2-torsion group. Hence,
$$
H_1\left(R/R\cap S,\, {\sf SP}^2\left(\frac{S\cap T}{(S\cap T)'}\right)\right)=H_1\left(R(S\cap T)/(S\cap T), \,{\sf SP}^2\left(\frac{S\cap T}{(S\cap T)'}\right)\right)
$$
is a 2-torsion group. Thus we have exact sequences
\begin{equation}\label{s2s2}
\xyma{H_3\left(\frac{R}{R\cap S},\, \tilde\otimes^2\left(\frac{S\cap T}{(R\cap S)(S\cap T)'}\right)\right)\ar@{->}[d]\\
H_2\left(\frac{R}{R\cap S},\, \frac{{\sf SP}^2(S\cap T/(S\cap T)')}{{\sf SP}^2(\frac{R\cap S}{R\cap (S\cap T)'})+K}\right)\ar@{->}[d]\ar@{->}[r] &  H_1\left(\frac{R}{R\cap S}, \,{\sf SP}^2\left(\frac{R\cap S}{R\cap (S\cap T)'}\right)\right)\ar@{->}[r]& (2-{\sf torsion})\\
H_2\left(\frac{R}{R\cap S},\, \frac{S\cap T}{(S\cap T)'}\otimes \frac{S\cap T}{(R\cap S)(S\cap T)'}\right)\ar@{=}[r] & H_4\left(\frac{R}{R\cap S}, \frac{S\cap T}{(R\cap S)(S\cap T)'}\right).
}
\end{equation}
Observe that, $\tilde\otimes^2\left(\frac{S\cap T}{(R\cap S)(S\cap T)'}\right)$ and $\frac{S\cap T}{(R\cap S)(S\cap T)'}$ are trivial $\mathbb Z[R/R\cap S]$-modules and one can use the Universal Coefficient Theorem to decompose the homology groups in the last diagram. By hypothesis, the left hand terms in the last diagram are torsion and therefore the left hand term in (\ref{epi10}) is a torsion group and so the proof is complete. $\Box$
\par\vspace{.5cm}

\section{The subgroup $D(F,\,{\bf rsf})$}\par\vspace{.5cm}
Let us consider the case $T=F$, and the corresponding subgroup $D(F,\,{\bf rsf})$.
\begin{theorem}
If the cohomological dimension  $cd(R/R\cap S)\leq 3$, then
$$
\frac{D(F,\,{\bf rsf})}{(R\cap S')'[(R\cap S)', R]}\hookleftarrow  H_3(R/R\cap S)\otimes \tilde\otimes^2\left(\frac{S}{(R\cap S)S'}\right).
$$

\end{theorem}[Recall that  $\tilde\otimes^2$ denotes  the antisymmetric tensor square].
\begin{proof} Let us set
$$
B:=\Delta^2(R\cap S)\Delta(R)+\Delta(R)\Delta^2(R\cap S)+\Delta(R\cap S)\Delta(R\cap S').
$$
Clearly, $B\subseteq \Delta(R)\Delta(S)\Delta(F)$, and we  have a natural commutative  diagram
$$
\xyma{\frac{F\cap (1+B)}{I(R,\,S,\,F)}\ar@{>->}[r]\ar@{>->}[d] & \frac{(R\cap S)'}{I(R,\,S,\,F)}\ar@{=}[d]\ar@{->}[r] & \frac{\Delta(R)\Delta(S)}{B}\ar@{->>}[d]\\
\frac{D(F,\,{\bf rsf})}{I(R,\,S,\,F)}\ar@{>->}[r] & \frac{(R\cap S)'}{I(R,\,S,\,F)}\ar@{->}[r] & \frac{\Delta(R)\Delta(S)}{\Delta(R)\Delta(S)\Delta(F).}}
$$
The sequence (\ref{3step}) implies that there is an exact sequence\\
\begin{multline}
H_1\left(R/R\cap S, \frac{R\cap S}{R\cap S'}\otimes \frac{R\cap S}{(R\cap S)'} \right)\to \\\\
H_1\left(R/R\cap S, \,{\sf SP}^2\left(\frac{R\cap S}{R\cap S'}\right)\right)\to\frac{F\cap (1+B)}{(R\cap S')'[(R\cap S)',\,R]}\to 0
\end{multline}\\
Using dimension shifting, we get \\
$$
H_1\left(R/R\cap S, \,\frac{R\cap S}{R\cap S'}\otimes \frac{R\cap S}{(R\cap S)'} \right)=H_3\left(R/R\cap S,\,  \frac{R\cap S}{R\cap S'}\right).
$$\\
The short exact sequence (\ref{ses}) implies the exact sequence\\
$$
H_4\left(R/R\cap S,\, \frac{S}{(R\cap S)S'}\right)\to H_3\left(R/R\cap S,\,  \frac{R\cap S}{R\cap S'}\right)\to H_3\left(R/R\cap S, \, S/S'\right).
$$\\
Again by dimension shifting, $$H_3\left(R/R\cap S,\,  S/S'\right)=H_3(RS/S, S/S')=H_5(R/R\cap S)=0.$$ Therefore, by the hypothesis on cohomological condition, both sides in the last exact sequence are zero. Consequently,
$$
H_1\left(R/R\cap S,\, {\sf SP}^2\left(\frac{R\cap S}{R\cap S'}\right)\right)=\frac{F\cap (1+B)}{(R\cap S')'[(R\cap S)',\,R]}
$$
Invoking the sequences (\ref{s2s2}), we get an isomorphsim
$$
H_3\left(R/R\cap S, \,\tilde\otimes^2\left(\frac{S}{(R\cap S)S'}\right)\right)=H_1\left(R/R\cap S,\, {\sf SP}^2\left(\frac{R\cap S}{R\cap S'}\right)\right).
$$
The result now follows from the Universal Coefficient Theorem.
\end{proof}\par\vspace{.5cm}\noindent
\begin{remark}
Decomposition of  the antisymmetric tensor square as $0\to -\otimes \mathbb Z/2\to \tilde\otimes^2\to \Lambda^2\to 0$, leads to the following diagram
$$
\xyma{H_3(R/R\cap S)\otimes \frac{S}{(R\cap S)S'}\otimes \mathbb Z/2\ar@{->}[d]\\
H_3(R/R\cap S)\otimes \left(\tilde\otimes^2\left(\frac{S}{(R\cap S)S'}\right)\right)\ar@{>->}[r]\ar@{->>}[d] & H_3\left(R/R\cap S, \tilde\otimes^2\left(\frac{S}{(R\cap S)S'}\right)\right)\ar@{->>}[r] & ({\sf torsion})\\ H_3(R/R\cap S)\otimes \Lambda^2\left(\frac{S}{(R\cap S)S'}\right).}
$$
We can thus conclude that
$$
{\sf tf}\left(H_3(R/R\cap S)\otimes \Lambda^2\left(\frac{S}{(R\cap S)S'}\right)\right)={\sf tf}\left(H_1\left(R/R\cap S,\, {\sf SP}^2\left(\frac{R\cap S}{R\cap S'}\right)\right)\right).
$$
Here $\sf tf$ means the torsion-free rank of an abelian group. \end{remark}

\subsection*{Example.}\par\vspace{.5cm}
Let $F=F(x_1,\,x_2,\,x_3,\,x_4,\,x_5)$,
\begin{align*}
& R:=\langle x_1,\,x_2,\,x_3\rangle^F,\\
& S:=\langle [x_1,\,x_2],\, [x_2,\,x_3],\,[x_1,\,x_3],\, x_4,\,x_5\rangle^F.
\end{align*}
The quotient group $R/R\cap S$ is then a free abelian group of rank three, with the images of  $x_1,\,x_2,\,x_3$ as generators, the group $\frac{S}{(R\cap S)S'}$ is a free abelian group of rank two with the images of  $x_4,\,x_5$ as generators. Therefore, \\
$$
H_3(R/R\cap S)\otimes \Lambda^2\left(\frac{S}{(R\cap S)S'}\right)\cong \mathbb Z.
$$\\
Hence, the quotient
\begin{equation}\label{example}
\frac{D(F,\,{\bf rsf})}{I(R,\,S,\,F)}
\end{equation}
is non-zero.
\par\vspace{.5cm}

At the moment we are not able to indicate explicitly the elements leading to the non-triviality of the quotient (\ref{example}). Here we present candidates for such elements. The recipe given below shows how the elements from the third homology can be used for constructing elements belonging to  generalized dimension subgroups.
\par\vspace{.5cm}
Consider a free group $F$ and elements $r_i\in R,\,t_i\in R\cap
S$  such that
$$ \prod_i[r_i,\,t_i]=1.
$$
Such elements come from the third homology:
$$
H_3(R/R\cap S)=ker\{R\wedge (R\cap S)\buildrel{[,]}\over\to R\}.
$$
Here the sign $\wedge$ means the non-abelian exterior product in the sense of Brown-Loday \cite{BL}.

For $d,\,e\in S$, consider the element
$$w=\prod_i[[r_i^{-1},\,d],\,[t_i,e]]$$
\par\vspace{.25cm}
\begin{prop}
$
w-1\in \Delta(R)\Delta(S)\Delta(F).
$
\end{prop}
\par\vspace{.25cm}
\begin{proof}
Working modulo $\Delta(R)\Delta(S)\Delta(F)$, we get
\begin{multline*}
1-w_i\equiv
(1-[r_i^{-1},\,d])(1-[t_i,\,e])-(1-[t_i,\,e])(1-[r_i^{-1},\,d])\equiv\\
-(1-[t_i,\,e])(1-[r_i^{-1},\,d])\equiv
-(1-[t_i,\,e])d^{-1}r_i((1-r_i^{-1})(1-d)-(1-d)(1-r_i^{-1}))\equiv\\
(1-[t_i,\,e])(1-r_i)(1-d)\equiv
e^{-1}t_i^{-1}((1-t_i)(1-e)-(1-e)(1-t_i))(1-r_i)(1-d)]\equiv\\
(1-e^{-1})(1-t_i)(1-r_i)(1-d)\equiv\\
-(1-e^{-1})r_it_i(1-[r_i,\,t_i])(1-d)\equiv
-(1-e^{-1})(1-[r_i,\,t_i])(1-d).
\end{multline*}
Clearly,
$$
1-w\equiv \sum_i(1-w_i).
$$
Therefore,
\begin{multline*}
1-w\equiv -(1-e^{-1})(\sum_i(1-[r_i,t_i]))(1-d)\equiv
-(1-e^{-1})(1-\prod_i[r_i,\,t_i])(1-d)=0.
\end{multline*}
\end{proof}

\par\vspace{.5cm}

\section{A combinatorial proof of a result of St\"ohr}\par\vspace{.5cm}
As mentioned earlier, the normal subgroup $D(F,\,{\bf rfr})$, which is a special case of Gupta's three subgroup problem,  has been identified by St\"{o}hr \cite{Stohr:1984}.   The following result  on free group rings is implicit  in this work based on using homological mehods.\par\vspace{.25cm}
\begin{theorem}\label{Stohr}
Let $F$ be a free group, and $R$ a  normal subgroup of $F$. If $a\in D(F, \,{\bf
frf})$, then $a^2\in D(F,\,{\bf rrf}+{\bf frr})$.\end{theorem}\par\vspace{.25cm}\noindent
In this section we give a combinatiorial proof of the above result, and bring out  the possibility of higher dimensional variations  of its statement.
 Before proceeding further, let us  give a sketch of the main steps from \cite{Stohr:1984} which yield Theorem \ref{Stohr}.\par\vspace{.5cm}
First identify the tensor square of the relation module $\bar{R}:=R/[R,\,R]$ as
$$
\bar{R}^{\otimes 2}= \frac{\bf rr}{\bf frr}
$$
and observe that there is a natural $\mathbb Z/2$-action on
$\bar{R}^{\otimes 2}$, namely the one  which permutes the factors. This action extends to the
natural quotient $$(\bar{R}^{\otimes 2})_F=\frac{\bf rr}{{\bf frr}+{\bf
rrf}}.$$ One of the main statements in \cite{Stohr:1984} is that
the $\mathbb Z/2$-action on the subgroup
$$
H_4(G)=\frac{{\bf rr}\cap {\bf frf}}{{\bf rrf}+{\bf frr}}\subseteq
\frac{\bf rr}{{\bf frr}+{\bf rrf}}
$$
is trivial. The proof in \cite{Stohr:1984}, which is homological,  uses
comparison of different projective resolutions. Let $a\in D(F,\,{\bf frf})$.
Then $a\in R',$ since ${\bf frf}\subset {\bf fr}$ and $D(F,\,{\bf fr})=R'$. The $\mathbb Z/2$-action which permutes the
terms in $(\bar{R}^{\otimes 2})_F$ sends
$$
a-1+{\bf frr}+{\bf rrf} \mapsto a^{-1}-1+{\bf frr}+{\bf rrf}.
$$
Since $a\in D(F,\,{\bf frf})$ and $a\in R'$, $a\in D(F,\,{\bf rr}\cap {\bf
frf})$. We conclude that
$$
a\equiv a^{-1}\mod {\bf frr}+{\bf rrf}
$$
and therefore
$$
a^2\in D(F,\,{\bf frr}+{\bf rrf}).
$$
\par\vspace{.25cm}\noindent
Since the above conclusion is a result purely in group rings, the following questions arise naturally. \par\vspace{.25cm}
\begin{itemize}
\item
Does there exist a
combinatorial proof of the above fact without the use of homology?\\
\item
Is it possible to generalize this result to more complicated
ideals and generalized dimension subgroups?\end{itemize}
\par\vspace{.25cm}\noindent
We answer the first question affirmatively, and offer some remarks on the second question.

\par\vspace{.5cm}
Let $G=F/R$, and choose a transversal $\{w(g)\}_{g\in G}$ for $G$ in
$F$:
$$
w(g)\mapsto g,\ g\in G, w(g)\in F,\ \text{with}\ w(1)=1.
$$
Then we have a function $W: G\times G\to R$ defined by
$$
w(g)w(h)=w(gh)W(g,\,h).
$$
This function satisfies a 2-cocycle condition
\begin{equation}\label{eq1} W(gh,\,k)W(g,\,h)^{w(k)}=W(g,\,hk)W(h,\,k)\
\text{for all}\ g,\,h,\,k\in G.
\end{equation}
\par\vspace{.25cm}
\begin{lemma}\label{lem1}
For $g,\,h\in G$, $(W(g,\,h)^{-1})^{w(gh)^{-1}}W(h^{-1},\,g^{-1})\in
[R,\,R].$
\end{lemma}\par\vspace{.25cm}
\begin{proof}
We have
\begin{align}
& w(h^{-1})w(g^{-1})=w(h^{-1}g^{-1})W(h^{-1}, \,g^{-1}),\ g,\,h\in
G,\label{eq2}\\
& w(g^{-1})=w(g)^{-1}W(g,\,g^{-1})=W(g^{-1}, \,g)w(g)^{-1},\ g\in G,\\
& W(g,\,g^{-1})^{w(g)}=W(g^{-1},\,g).
\end{align}
Substituting in (\ref{eq2}) we, in turn, have
\begin{align*}
& w(h)^{-1}W(h,\,h^{-1})w(g)^{-1}W(g,\,g^{-1})=w(gh)^{-1}W(gh,\,
(gh)^{-1})W(h^{-1},\, g^{-1})\\
&
w(h)^{-1}w(g)^{-1}W(h,h^{-1})^{w(g)^{-1}}W(g,\,g^{-1})=w(gh)^{-1}W(gh,\,
(gh)^{-1})W(h^{-1}, \,g^{-1})\\
&
W(g,\,h)^{-1}w(gh)^{-1}W(h,\,h^{-1})^{w(g)^{-1}}W(g,g^{-1})=w(gh)^{-1}W(gh,
(gh)^{-1})W(h^{-1}, g^{-1})
\end{align*}
Finally, we have the following relation:
\begin{equation}\label{eq5}
W(g,\,h)^{-w(gh)^{-1}}W(h,h^{-1})^{w(g)^{-1}}W(g,\,g^{-1})=W(gh,\,
(gh)^{-1})W(h^{-1},\, g^{-1}).
\end{equation}
Equation (\ref{eq1}) implies that
\begin{multline}\label{eq6}
W(gh,\,h^{-1}g^{-1})W(g,\,h)^{w(h^{-1}g^{-1})}=W(g,\,g^{-1})W(h,h^{-1}g^{-1})=\\
W(g,\,g^{-1})W(h,\,h^{-1})^{w(g^{-1})}W(h^{-1},\,g^{-1}).
\end{multline}
Substituting the value of $W(gh, \,h^{-1}g^{-1})$ from (\ref{eq5})
to (\ref{eq6}), we get
$$
[(W(g,\,h)^{-1})^{w(gh)^{-1}}W(h^{-1},\,g^{-1})]^2\in [R,\,R].
$$
Since $R/[R,\,R]$ is free abelian, we conclude that
$$(W(g,\,h)^{-1})^{w(gh)^{-1}}W(h^{-1},g^{-1})\in [R,\,R].$$
\end{proof}

Now we are ready to prove Theorem \ref{Stohr} using
only combinatorial tools, and without homological algebra.  In fact, we have the following  more general result, from which Theorem  \ref{Stohr}   follows in case $S=T=F$.\par\vspace{.5cm}\noindent
\begin{theorem}\label{Stohr1}
Let $R,\,S,\,T$ be normal subgroups in $F$, such that $R\subseteq
S,\,T$. If
$$
a\in D(F,\,{\bf srt+trs}),
$$
then
$$
a^2\in D(F,\,{\bf rrs+srr+trr+rrt}).
$$
\end{theorem}
\begin{proof} Let $a\in  D(F,\,{\bf srt+trs})$ so that we have an expression
$$
1-a=\sum (1-f)(1-r)(1-g)+\sum (1-f')(1-r')(1-g'),
$$
where the two sums are  taken over products with  $f, \,g'\in S,\, g,\,f'\in T,\ r,\,r'\in R$.
On opening the brackets, we get \begin{equation}\label{sum} 1-a=\sum
(1-f-r-g+fr+fg+rg-frg)+\sum (1-f'-r'-g'+f'r'+f'g'+r'g'-f'r'g').
\end{equation}
We pick a set of representatives $\{w(g)\}_{g\in F/R}$ in $F$ for the elements of the quotient group $F/R$.
Then every element ${\sf w}\in F$ can be written uniquely as
${\sf w}=w(\bar{{\sf w}})r_{\sf w}$, with $\bar{{\sf w}}={\sf w}R$ and $r_{\sf w}\in R$, and so we have
 \begin{align*} & f=w(\bar f)r_f, & & f'=w(\bar f')r_{f'},\\ & g=w(\bar g)r_g, &  & g'=w(\bar g')r_{g'},\\
&
fr=w(\overline{fr})r_fr, & & f'r'=w(\overline{f'r'})r_{f'}r'\\
& fg=w(\overline{fg})W(f,g)r_f^gr_g, & & f'g'=w(\overline{f'g'})W(f',g')r_{f'}^{g'}r_{g'}\\
& rg=w(\overline{rg})r_gr^g, & &  r'g'=w(\overline{r'g'})r_{g'}{r'}^{g'},\\
& frg=w(\overline{frg})W(f,g)r_f^gr_gr^g, & &
f'r'g'=w(\overline{f'r'g'})W(f',g')r_{f'}^{g'}r_{g'}{r'}^{g'}.
\end{align*}
Let $\pi:F\mapsto R$ be the projection given  by ${\sf w}\mapsto
r_{\sf w}$. Since the element $a$ lies in $[R,\,R]$, $\pi(a)=a$.
The first sum in (\ref{sum}),  projects under $\pi$ to the
following sum
\begin{multline}
\sum (1-r_f-r-r_g+r_fr+W(f,g)r_f^gr_g+r_gr^g-W(f,g)r_f^gr_gr^g)=\\
\sum ((1-r_f)(1-r)-(1-W(f,g)r_f^g)r_g(1-r^g))
\end{multline}
Therefore, modulo ${\bf r}^2{\bf t}+{\bf tr}^2$ we have this sum
equivalent to
$$
\sum (1-W(f,g))(1-r^g)\equiv \sum (1-W(f,g)^{g^{-1}})(1-r).
$$
In the same way  we see
that, modulo ${\bf r}^2{\bf s}+{\bf sr}^2$,  the second sum in (\ref{sum}) is equivalent to the sum
$$\sum (1-W(f',\,g')^{g'^{-1}})(1-r').$$ Hence,
\begin{multline}\label{sum1}
1-a\equiv \sum (1-W(f,\,g)^{g^{-1}})(1-r)+\sum
(1-W(f',\,g')^{g'^{-1}})(1-r')\\
\mod {\bf r}^2{\bf s}+{\bf sr}^2+{\bf
r}^2{\bf t}+{\bf tr}^2.
\end{multline}
On applying the  involution $f\mapsto f^{-1}$ on $F$ to the equation  (\ref{sum}), we have
$$
1-a^{-1}=\sum (1-g^{-1})(1-r^{-1})(1-f^{-1})+\sum
(1-g'^{-1})(1-r'^{-1})(1-f'^{-1}).
$$
Repeating the same process as above, we get
\begin{multline}\label{sum2}
1-a^{-1}\equiv -\sum (1-W(g^{-1},\,f^{-1})^f)(1-r)-\sum
(1-W(g'^{-1},\,f'^{-1})^{f'})(1-r')\\ \mod {\bf r}^2{\bf s}+{\bf
sr}^2+{\bf r}^2{\bf t}+{\bf tr}^2.
\end{multline}
Subtracting $1-a^{-1}$ from $1-a$, by (\ref{sum1}) and (\ref{sum2}) we get
\begin{multline*}
1-a^2\equiv\sum (1-W(f,\,g)^{g^{-1}}W(g^{-1},\,f^{-1})^{f})(1-r)+
\\ (1-W(f',\,g')^{g'^{-1}}W(g'^{-1},\,f'^{-1})^{f'})(1-r')\\ \mod {\bf r}^2{\bf s}+{\bf
sr}^2+{\bf r}^2{\bf t}+{\bf tr}^2.
\end{multline*}
Lemma \ref{lem1} implies that
$$
W(f,\,g)W(g^{-1},\,f^{-1})^{fg}, W(f',\,g')W(g'^{-1},f'^{-1})^{f'g'}\in
[R,\,R].
$$
Hence  $1-a^2\in {\bf r}^2{\bf s}+{\bf sr}^2+{\bf
r}^2{\bf t}+{\bf tr}^2.$
\end{proof}

\begin{remark}  On invoking Theorem 4, \cite{Stohr:1984}, it follows that under the hypothesis of the above theorem,  $a^4\in [R,\,R,\,ST].$\end{remark}
\par\vspace{.25cm} One can prove the following  result by proceeding in a way similar to that for  the proof of the preceding theorem, and therefore we  omit the details.
\par\vspace{.25cm}\noindent
\begin{theorem} Let $R,\,S,\,T$ be normal subgroups of a free group $F$ with   $R\subseteq S,\,T$.
 If $
a\in D(F,\,{\bf srs+trt}),
$
then
$
a^2\in D(F,\,{\bf rrs+srr+rrt+trr})
$
and $a^4\in [R,\,R,\,ST]$.
\end{theorem}

\par\vspace{.5cm}
The general problem in free group rings, of which the foregoing are special cases, asks for the  identification of normal subgroups $D(F,\, \mathfrak a)$, where $\mathfrak a$ is a sum of ideals of the form ${\bf r_1\,\ldots\,r_n}$ with   $R_1,\,\ldots\,,R_n$  normal subgroups of the  given free group $F$.  As a contribution to this general problem,  we present the following two results.
\par\vspace{.25cm}
\begin{theorem} Let  $R\subseteq S,\ T$ be normal subgroups of a free group $F$. \\ If
$
a\in D(F,\,{\bf rsst+tssr}),
$
then
\begin{equation}\label{hp1}
a^2\in D(F,\,{\bf rsss+sssr+tsss+ssst}+([S,\,S]-1){\bf s})
\end{equation}
and
\begin{equation}\label{hp2}
a^6\in [S,\,S,\,S,\,RT].
\end{equation}\end{theorem}
\par\vspace{.15cm}
\begin{proof} While the proof of (\ref{hp1}) is similar to that of Theorem \ref{Stohr1}, and so we omit it, the assertion (\ref{hp2}) follows from (\ref{hp1}) and the following  general result:

\begin{multline}\label{ge}
\text{\it If $R$ is a normal subgroup of the free group $F$ and} \\
\text{\it $a\in D(F,\,{\bf rrrf}+{\bf frrr}+([R,\,R]-1){\bf r}),$
then
$a^3\in [R,\,R,\,R,\,F].$}   \end{multline}
\par\vspace{.5cm}
To prove (\ref{ge}) consider  the natural exact sequence
$$
\xyma{\gamma_3(R)/\gamma_4(R) \ar@{->}[r] \ar@{=}[d] & \frac{\bf rrr}{{\bf frrr}+([R,\,R]-1){\bf r}}\ar@{=}[d]\ar@{->>}[r] & \frac{\bf rrr}{{\bf frrr}+([R,\,R]-1){\bf r}+(\gamma_3(R)-1)\mathbb Z[R]}\ar@{=}[d]\\
\EuScript L^3(R_{ab})\ar@{->}[r] & {\sf SP}^2(R_{ab})\otimes R_{ab} \ar@{->>}[r] & {\sf SP}^3(R_{ab}),
}
$$
where $\EuScript L^3$ and ${\sf SP}^3$ are the third Lie and symmetric power functor respectively, and $R_{ab}$ is the abelianization of $R$. Applying the homology functor $H_*(F,\,-)$ to this sequence, we get a long exact sequence which connects $H_1$ and $H_0$, which, in turn,  implies that
$$
\frac{D(F,\, {\bf rrrf}+{\bf frrr}+([R,\,R]-1){\bf r})}{[R,\,R,\,R,\,F]}={\sf Coker}\{H_1(F,\,{\sf SP}^2(R_{ab})\otimes R_{ab})\to H_1(F, \,{\sf SP}^3(R_{ab}))\}.
$$
The assertion (\ref{ge})  follows from the simple fact that, for a free abelian $A$, the  natural composition
$$
{\sf SP}^3(A)\hookrightarrow {\sf SP}^2(A)\otimes A\twoheadrightarrow {\sf SP}^3(A)
$$
is  multiplication by 3.  \end{proof}

\par\vspace{.25cm}
Our concluding  result is a  generalization of Kanta Gupta's identification of $D(F, \,{\bf rffr})$ \cite{Kanta:1983}. \par\vspace{.25cm}
\begin{theorem}\label{thprod4}
If $R\subseteq S$ are normal subgroups of a free group $F$, then
$$
D(F,\, {\bf rsfr})=D(F, \,{\bf rfsr})=[R\cap S',\, R\cap S',\, R]\gamma_4(R).
$$
\end{theorem}\par\vspace{.25cm}
\begin{proof}
Observe that $D(F,\, {\bf rsfr})\subseteq D(F, {\bf rfr})=\gamma_3(R)$. Therefore, by Lemma \ref{onemorelemma},
$$
D(F,\, {\bf rsfr})=D(F,\, \Delta^4(R)+\Delta(R)\Delta(R\cap S')\Delta(R)).
$$
Recall from \cite{MP:2015a} that, for a free group $F$ and its normal subgroup $N$, if $(F/N)_{ab}$ is \linebreak 2-torsion-free, then
$$
\frac{D(F,\, {\bf fnf+f}^4)}{[N,\,N,\,F]\gamma_4(F)}\cong L_1{\sf SP}^3((F/N)_{ab}).
$$
In our situation, the quotient $R/R\cap S'\subseteq S/S'$ is torsion-free, therefore, the contribution from the derived functor $L_1{\sf SP}^3$  vanishes and the  result follows.
\end{proof}\par\vspace{.5cm}
Observe that, the condition $R\subseteq S$ in Theorem \ref{thprod4} significantly simplifies the identification.
\par\vspace{.5cm}{\it
For arbitrary normal subgroups $R,\ S$, we conjecture that}
\par\vspace{.25cm}
$$
D(F,\,{\bf rssr})=\sqrt{[\gamma_3(R\cap S), \,R][\gamma_2(R\cap
S'), R]}.
$$
\par\vspace{1cm}
\section*{Acknowledgement}
The research of the first author is supported by the Russian Science Foundation, grant
N 16-11-10073.
\newpage
\par\vspace{1cm}\noindent
Roman Mikhailov\\ St Petersburg
Department of Steklov Mathematical Institute\\and\\ Chebyshev
Laboratory, St Petersburg State University\\ 14th Line, 29b, Saint
Petersburg 199178 Russia\\
Email:\ romanvm@mi.ras.ru\par\vspace{.5cm}\noindent
Inder Bir S. Passi\\ Centre for Advanced Study in Mathematics\\ Panjab
University, Sector 14, Chandigarh 160014 India\\ and \\ Indian
Institute of Science Education and Research, Mohali (Punjab)
140306 India\\
Email:\ ibspassi@yahoo.co.in
\end{document}